\documentclass[11pt]{amsart}

\usepackage{graphicx}
\usepackage{amssymb,url}
\usepackage{epstopdf}

\newcommand{\NN}{\mathbb{N}}
\newcommand{\RR}{\mathbb{R}}
\newcommand{\ZZ}{\mathbb{Z}}
\newcommand{\TT}{\mathbb{T}}

\newcommand{\AAA}{\mathcal{A}}
\newcommand{\BBB}{\mathcal{B}}

\newcommand{\EEE}{\mathcal{E}}
\newcommand{\HHH}{\mathcal{H}}
\newcommand{\MMM}{\mathcal{M}}
\newcommand{\OOO}{\mathcal{O}}
\newcommand{\TTT}{\mathcal{T}}

\newcommand{\llim}{\varliminf}
\newcommand{\ulim}{\varlimsup}
\newcommand{\eps}{\varepsilon}
\newcommand{\ph}{\varphi}
\newcommand{\coloneqq}{\mathrel{\mathop:}=}

\DeclareMathOperator{\supp}{supp}

\newcommand{\foot}[1]{}

\numberwithin{equation}{section}

\theoremstyle{plain}
\newtheorem{theorem}{Theorem}[section]
\newtheorem{lemma}[theorem]{Lemma}

\theoremstyle{definition}
\newtheorem{definition}[theorem]{Definition}

\theoremstyle{remark}
\newtheorem{example}[theorem]{Example}
\newtheorem{exercise}{Exercise}[section]

\title{Measure Theory through Dynamical Eyes}
\author{Vaughn Climenhaga}
\address{Department of Mathematics \\ University of Houston \\ Houston, TX 77005, USA}
\email{climenha@math.uh.edu}

\author{Anatole Katok}
\address{Department of Mathematics \\ McAllister Building \\ Pennsylvania State University \\ University Park, PA 16802, USA}
\email{katok\_a@math.psu.edu}

\begin{document}

\maketitle

These notes are a somewhat embellished version of two rather informal evening review sessions given by the second author on July 14 and 15, 2008 at the 
Bedlewo summer school, 
 which provide a brief overview of some of the basics of measure theory and its applications to dynamics which are foundational to the various courses at the school.

A number of results are quoted without proof, or with at most a bare sketch of a proof; references are given where full proofs may be found. Similarly, most basic definitions are assumed to be known, and we defer their reiteration to the references.

In light of the above, we emphasise that this presentation is \emph{not} meant to be either comprehensive or self-contained; the reader is assumed to have some knowledge of the basic concepts of
measure theory, ergodic theory, and hyperbolic dynamics, which will appear without any formal introduction.  The tone is meant to be conversational rather than authoritative, and the goal is to give a general idea of various concepts which should eventually be examined thoroughly in the appropriate references.

For a full presentation of the concepts in \S1, which concerns abstract measure theory, we refer the reader to Halmos' book~\cite{pH78} (for more basic facts) and to  Rokhlin's article~\cite{vR62}.  The topics in measurable dynamics mentioned in \S2 receive a more complete treatment in a later article by Rokhlin~\cite{vR67}, and the account in \S3 of the relationship between foliations and measures in smooth dynamics draws on Barreira and Pesin's book~\cite{BP07}, along with two articles by Ledrappier and Young~\cite{LY85a,LY85b}.

The first author would like to thank Andrey Gogolev and Misha Guysinsky for providing useful references and comments.


\section{Abstract Measure Theory}

\subsection{Points, sets, and functions}

There are three ``lenses'' through which we can view measure theory; we may think of it in terms of \emph{points}, in terms of \emph{sets}, or in terms of \emph{functions}.  To put that a little more concretely, suppose we have a triple $(X,\TTT,\mu)$ comprising a measurable space, a $\sigma$-algebra, and a measure. Then we may focus our attention either on the space $X$ (and concern ourselves with points), or on the $\sigma$-algebra $\TTT$ (and concern ourselves with sets), or on the space $L^2(X,\TTT,\mu)$ (and concern ourselves with functions).

All three points of view play an important role in dynamics, and various definitions and results can be given in terms of any of the three. 
We will see later that the last two  are completely equivalent in the greatest generality but the first  requires certain additional, albeit natural,  assumptions explained  in Section~\ref{SLebesguespaces} (see Theorem~\ref{thm:Lebesgue2} and Definition \ref{defLebesgue}).
 Of particular interest to us will be the correspondence between partitions of the space $X$,  sub-$\sigma$-algebras of $\TTT$, and subspaces of $L^2(X,\TTT,\mu)$.

First let us consider the set of all partitions of $X$.  This is a partially ordered set, with ordering given by refinement; given two partitions $\xi$, $\eta$, we say that $\xi$ is a \emph{refinement} of $\eta$, written $\xi\geq \eta$, if and only if every $C\in\xi$ is contained in some $D\in\eta$.  In this case, we also say that $\eta$ is a \emph{coarsening} of $\xi$.  The finest partition (which in this notation may be thought of as the ``largest'') is the partition into points, denoted $\eps$, while the coarsest (the ``smallest'') is the trivial partition $\{ X \}$, denoted $\nu$.

As on any partially ordered set, we have a notion of \emph{join} and \emph{meet}, corresponding to least upper bound and greatest lower bound, respectively.  Following~\cite{vR67}, we shall refer to these as the \emph{product} and \emph{intersection}, and we briefly recall their definitions. Given two partitions $\xi$ and $\eta$, their product (join) is
\begin{equation}\label{eqn:join}
\xi \vee \eta \coloneqq \{\, C \cap D \mid C \in\xi,\ D\in\eta\,\}.
\end{equation}
This is the coarsest partition which refines both $\xi$ and $\eta$, and is also sometimes referred to as the \emph{joint
partition}. The intersection (meet) of $\xi$ and $\eta$ is the finest partition which coarsens both $\xi$ and $\eta$, and is denoted $\xi\wedge\eta$; in general, there is no analogue of~\eqref{eqn:join} for $\xi\wedge\eta$.

So much for partitions; what do these have to do with $\sigma$-algebras, or with $L^2$-spaces?  Given a partition $\xi$, we may consider the collection of all measurable subsets $A\subset X$ which are unions of elements of $\xi$; this collection forms a sub-$\sigma$-algebra of $\TTT$, which we denote by $\BBB(\xi)$.  We will see later that this correspondence is far from injective; for example, certain partitions whose elements are countable sets are associated with the  trivial $\sigma$-algebra, see Example~\ref{eg:nonmeas}.

 Similarly, we may consider the collection of all square integrable  functions which are constant on elements of $\xi$; this collection (more precisely, the collection of equivalence classes of such functions) forms the subspace $L^2(X,\BBB(\xi),\mu) \subset L^2(X,\TTT,\mu)$.

\begin{example}\label{eg:lines}
Let $X=[0,1]\times [0,1]$ be the unit square with Lebesgue measure $\lambda$, and $\TTT$ the Borel $\sigma$-algebra.  Let $\xi=\{\, \{x\} \times [0,1] \mid x\in [0,1] \,\}$ be the partition into vertical lines; then $\BBB(\xi)$ is the sub-$\sigma$-algebra consisting of all sets of the form $E\times [0,1]$, where $E\subset [0,1]$ is Borel, and $L^2(X,\BBB(\xi),\lambda)$ is the space of all square-integrable functions which depend only on the $x$-coordinate.  The latter is canonically isomorphic to the space of square-integrable functions on the unit interval.
\end{example}

The class of sub-$\sigma$-algebras of $\TTT$ and the class of subspaces of $L^2(X,\TTT,\mu)$ are both partially ordered by containment, and this partial ordering is preserved by the correspondences described above.  Thus the map $\xi \mapsto \BBB(\xi)$ is a morphism of partially ordered sets; it is natural to ask whether this morphism is injective (and hence invertible) on a certain class of partitions, and we will return to this question eventually.  First, however, we turn to the question of classifying measure spaces, and hence the associated classes of partitions and $\sigma$-algebras, since the end result turns out to be relatively simple.

\subsection{Lebesgue spaces}\label{SLebesguespaces}

It is a somewhat serendipitous fact that although one may consider many different measure spaces $(X,\TTT,\mu)$, which on the face of it are quite different from each other, all of the examples in which we will be interested actually fall into a relatively simple classification.

To elucidate this statement, let us consider the two fundamental examples of measure spaces.  The simplest sort of measure space is an \emph{atomic} space, in which $X$ is a finite or countable set, $\TTT$ is the entire power set of $X$, and $\mu$ is defined by the sequence of numbers $\mu(x_i), x_i\in X$.  Each of the points $x_i$ is an \emph{atom} -- that is, a measurable set $A$ of positive measure such that every subset $B\subset A$ has either $\mu(B)=0$ or $\mu(B)=\mu(A)$.  Atomic spaces are discrete objects, which belong to combinatorics as much as to measure theory, and do not require the full power of the latter theory.

At the other end of the spectrum stand the \emph{non-atomic spaces}, in which every set of positive measure can be decomposed into two subsets of smaller positive measure.  The easiest example of such a space is the interval $[0,1]$ with Lebesgue measure, where the $\sigma$-algebra $\TTT$ is the collection of Lebesgue sets.  In fact, up to isomorphism and setting aside examples which are for our purposes pathological, this is the \emph{only} example of such a space.



What does it mean for measure spaces to be isomorphic?  The most immediate (and essentially correct) idea is to require  existence of a bijection between the spaces which carries measurable sets into  measurable sets both ways and preserves the 
measure.  There are some technicalities  related with   different ways of defining 
$\sigma$-algebras of measurable sets (e.g.\ Borel  or Lebesgue on the interval)  and also ignoring some   ``bad'' sets of measure zero (e.g.\  cardinality of the space may be artificially increased by adding a set of points of measure zero of large cardinality). These inessential problems aside,   there are  examples of  isomorphism which  look striking on the surface. 

\begin{example}\label{Peano} Consider   the unit interval $I$  and the unit square  $I\times I$  with Lebesgue measure.  The standard construction of the {\em Peano curve}  with division of  $I$ into $4^n$ basic intervals $\left(\frac a{4^n}, \frac{a+1}{4^n}\right)$ and $I\times I$ into $4^n$ equal squares  provides a continuous surjective map $f\colon I\to I\times I$   which preserves the measure of any union of basic intervals and hence   of any measurable set. While this map is obviously not bijective,  it is  bijective between complements of   certain measure zero sets: namely, $A\subset I$, the union of endpoints of all basic intervals, and $B\subset I\times I$, the union of  boundaries of all squares  involved in the construction. 
\end{example}

To discuss the general case, we first need some definitions.

\begin{definition}\label{def:sep}
Introduce a pseudo-metric on the (measured) $\sigma$-algebra\footnote{Because we deal with measure spaces (not just measurable spaces), we consider not just the $\sigma$-algebra $\TTT$, but also the measure $\mu$ it carries.  This will be implicit in our discussion throughout this section.} $\TTT$ by the formula
\[
d_\mu(A,B)=\mu(A \bigtriangleup B),
\]
where $A\bigtriangleup B$ denotes the symmetric difference $(A\cup B) \setminus (A\cap B)$.  We say that two sets $A,B\in \TTT$ are \emph{equivalent mod zero} if $d_\mu(A,B)=0$, and write $A\circeq B$.

If we pass to the quotient space $\TTT / \circeq$, we obtain a true metric space; we say that the $\sigma$-algebra is \emph{separable} if this metric space is separable.  That is, $\TTT$ is separable if and only if there exists a countable set $\{A_n\}_{n\in \NN}$ which is dense in $\TTT$ with the $d_\mu$ pseudo-metric.

The \emph{completion} of $\TTT$ is the $\sigma$-algebra generated by $\TTT$ together with all subsets of null sets (that is, sets $A$ with $\mu(A)=0$).  
We say that two $\sigma$-algebras $\TTT,\TTT'$ are equivalent mod zero if they have the same completion, and write $\TTT \circeq \TTT'$.
\end{definition}

\begin{exercise}\label{ex:sep2}
Recall that if $\AAA\subset \TTT$ is any collection of sets, then $\sigma(\AAA)$, the $\sigma$-algebra generated by $\AAA$, is the smallest $\sigma$-algebra that contains $\AAA$.  Show that if $\TTT$ is equivalent mod zero to a $\sigma$-algebra generated by a countable collection of sets, then $\TTT$ is separable.
\end{exercise}

In~\cite{vR62}, a stronger definition of separable is used, which implies the condition in Exercise \ref{ex:sep2}.  We follow the definition in~\cite{pH78}.  When the measure is non-atomic, the two definitions are equivalent.

\begin{exercise}\label{ex:sep3}
Show that if $X$ is a separable metric space, $\TTT$ is the $\sigma$-algebra of Borel subsets of $X$, and $\mu$ is any probability measure, then $\TTT$ is separable as in Definition~\ref{def:sep}.
\end{exercise}

For the sake of simplicity in what follows, we will always consider sets (of positive measure)  and $\sigma$-algebras up to equivalence mod zero (in particular, we will not distinguish between a $\sigma$-algebra and its completion), and will write $=$ in place of $\circeq$.

We say that two $\sigma$-algebras $\TTT,\TTT'$ are \emph{isomorphic} if there exists a bijection $\rho\colon \TTT \to \TTT'$ which preserves measure:
\begin{equation}\label{eqn:meas-pres}
\mu'(\rho(E)) = \mu(E)
\end{equation}
and which respects unions and complements (and hence intersections as well)
\begin{align}\label{eqn:union}
\rho \left( \bigcup_{n=1}^\infty E_i \right) &= \bigcup_{n=1}^\infty \rho(E_i), \\
\label{eqn:comp}
\rho(X \setminus E) &= X' \setminus \rho(E).
\end{align}
Obviously  one way (but not the only way) to produce an isomorphism of $\sigma$-algebras is to have 
an isomorphism $\pi$ between measure spaces as in Example~\ref{Peano} and then to set $\rho(E) = \pi^{-1}(E)$.

Up to this notion of isomorphism, all the $\sigma$-algebras with which we are concerned can be studied by one ``master'' example:\footnote{A state of affairs which Tolkein would surely render quite poetically, particularly if we were to adopt a slightly different line of exposition and consider \emph{$\sigma$-rings} instead of $\sigma$-algebras.}

\begin{theorem}\label{thm:Lebesgue}
A separable (measured, complete) $\sigma$-algebra with no atoms is isomorphic to the $\sigma$-algebra of Lebesgue sets on the unit interval.
\end{theorem}
\begin{proof}
For complete details of the proof, see \cite[Section 41, Theorem C]{pH78}.  Here we give the main ideas.

Given a countable collection of sets $A_n\in \TTT$, write $A_n^0 = A_n$ and $A_n^1 = X\setminus A_n$, and let $\xi_n=\bigvee_{i=1}^n \{A_n^0, A_n^1\}$. This is an  increasing sequence of partitions of $X$, and because $\TTT$ is separable, we may take the sets $A_n$ to be such that $\hat\TTT := \bigcup_{n\geq 1} \BBB(\xi_n)$ is dense in $\TTT$.  Moreover, $\xi_n$ has $\leq 2^n$ elements, which can be indexed by words $w = w_1\cdots w_n \in \{0,1\}^n$ as follows:
\[
A_w = A_1^{w_1} \cap A_2^{w_2} \cap \cdots \cap A_n^{w_n}.
\]
Write $\TTT'$ for the $\sigma$-algebra of Lebesgue sets on the unit interval.  We can define a map $\rho \colon \hat\TTT \to \TTT'$ as follows.
\begin{enumerate}
\item For a fixed $n$, order the elements of $\xi_n$ lexicographically: for example, with $n=3$ we have
\[
A_{000} \prec A_{001} \prec A_{010} \prec A_{011} \prec \cdots \prec A_{111}.
\]
\item Identify these sets with subintervals of $[0,1]$ with the same measure and the same order:  thus
\begin{align*}
\rho(A_{000}) &= [0,\mu(A_{000})), \\
\rho(A_{001}) &= [\mu(A_{000}),\mu(A_{000})+\mu(A_{001})),
\end{align*}
and so on, as shown in Figure~\ref{fig:lebesgue}.  In general we have
\begin{equation}\label{eqn:rhoAw}
\rho(A_w) = \left[\sum_{v\prec w} \mu(A_v), \sum_{v\preceq w} \mu(A_v)\right),
\end{equation}
where the sums are over words $v$ of the same length as $w$.  Note that the image is empty whenever the sums are equal, which happens exactly when $\mu(A_w)=0$.  In Figure~\ref{fig:lebesgue}, for example, we have $\rho(A_{011}) = \rho(A_{111}) = \emptyset$.
\end{enumerate}

\begin{figure}[htbp]
\includegraphics{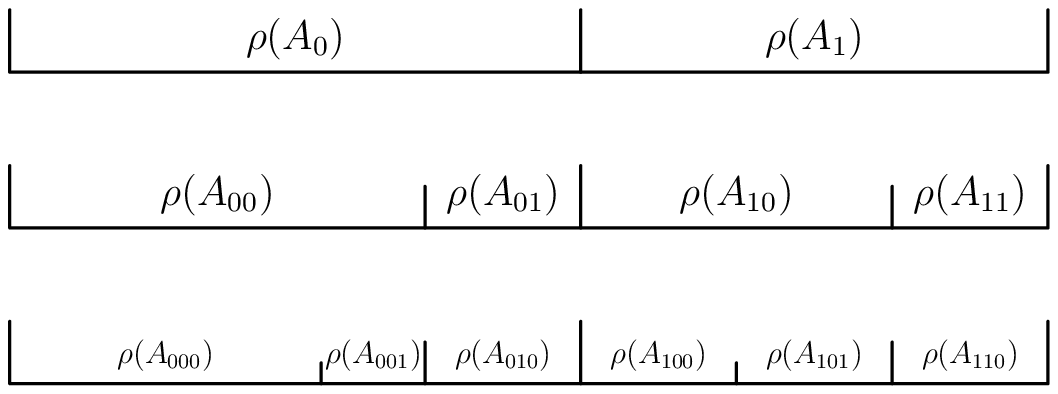}
\caption{Defining a map $\rho\colon \hat\TTT \to \TTT'$.}
\label{fig:lebesgue}
\end{figure}

It is clear from the construction that $\rho$ satisfies~\eqref{eqn:union} whenever the union is an element of $\hat\TTT$.  Furthermore, $\rho$ is an isometry with respect to the metrics $d_\mu$ and $d_\lambda$, where $\lambda$ is Lebesgue measure on the interval, and so it can be extended from the dense set $\hat\TTT$ to all of $\TTT$.  The identities~\eqref{eqn:meas-pres} and~\eqref{eqn:comp} hold on  $\hat\TTT$, and extend to $\TTT$ by continuity.

Finally, because $\TTT$ is non-atomic, the measure of the sets in $\xi_n$ goes to 0, and so the length of their images under $\rho$ goes to 0 as well, which shows that $\rho(\TTT) = \TTT'$. 
\end{proof}

\begin{exercise}\label{ex:lex-sum}
Show that the sum appearing in~\eqref{eqn:rhoAw} can also be written as
\[
\sum_{v\prec w} \mu(A_v) = \sum_{\stackrel{1\leq k\leq n}{w_k = 1}} \mu(A_{w_1\cdots w_{k-1} 0}),
\]
where $n$ is the length of the word $w$.
\end{exercise}

The proof of Theorem~\ref{thm:Lebesgue} describes the construction of a \emph{$\sigma$-algebra} isomorphism.  At a first glance, it may appear that this also creates an isomorphism between the measure spaces themselves, but in fact, the proof as it stands may not produce such an isomorphism.  
The problem  is the potential presence of ``holes''  in the  space $X$  distributed among its points in a non-measurable way. 
In order to give a classification result for measure spaces themselves, we need one further condition in addition to separability which prevents the appearance of such ``holes''. 


Let $\xi_1 = \{C_1, \dots, C_n\}$ be a finite partition of $X$ into measurable sets, and let $\TTT_1 = \BBB(\xi_1)$ be the $\sigma$-algebra which contains all unions of elements of $\xi_1$, so that $\TTT_1$ contains $2^n$ sets.  Partitioning each $C_i$ into $C_{i,1},\dots,C_{i,k_i}$, we obtain a finer partition $\xi_2$ and a larger $\sigma$-algebra $\TTT_2 = \BBB(\xi_2)$ whose elements are unions of none, some, or all of the $C_{i,j}$. Iterating this procedure, we have a
sequence of partitions
\begin{equation}\label{eqn:parts}
\xi_1 < \xi_2 < \cdots,
\end{equation}
each of which is a refinement of the previous partition, and a sequence of $\sigma$-algebras
\begin{equation}\label{eqn:salgs}
\TTT_1 \subset \TTT_2 \subset \cdots.
\end{equation}
This is obviously reminiscent of the construction from the proof of Theorem~\ref{thm:Lebesgue}. 
An even more useful image to keep in mind here is the standard picture of the construction of a Cantor set, in which the unit interval is first divided into two pieces, then four, then eight, and so on -- these ``cylinders'' (to use the terminology arising from symbolic dynamics) are the various sets $C_i$, $C_{i,j}$, etc.

We may consider the ``limit'' of the sequence~\eqref{eqn:parts}:
\begin{equation}\label{eqn:xifromxin}
\xi = \bigvee_{n=1}^\infty \xi_n.
\end{equation}
Each element of $\xi$ corresponds to a ``funnel''
\begin{equation}\label{eqn:funnel}
C_{i_1} \supset C_{i_1,i_2} \supset C_{i_1,i_2,i_3} \supset \cdots
\end{equation}
of decreasing subsets within the sequence of partitions; the intersection of all the sets in such a funnel is an element of $\xi$.

The sequence~\eqref{eqn:parts} is a \emph{basis} if it generates both the $\sigma$-algebra $\TTT$ and the space $X$, as follows:
\begin{enumerate}
\item the associated $\sigma$-algebras $\TTT_n
\coloneqq \BBB(\xi_n)$ from~\eqref{eqn:salgs} have the property that $\bigcup_{n\geq 1} \TTT_n$ generates $\TTT$;
\item it generates the space $X$; that is, every ``funnel'' $C_{i_1} \supset C_{i_1,i_2} \supset \cdots$ as in~\eqref{eqn:funnel} has intersection containing at most one point.  Equivalently, any two points $x$ and $y$ are separated by some partition $\xi_n$, and so $\xi \coloneqq \bigvee_{n=1}^\infty \xi_n = \eps$, the partition into points.
\end{enumerate}


Note that the existence of an increasing sequence of finite or countable partitions satisfying~(1) is equivalent to separability of the
$\sigma$-algebra.

It is often convenient to choose a sequence $\xi_n$ such that at each stage, each cylinder set $C$ is partitioned into
exactly two smaller sets. This gives a one-to-one correspondence between sequences in $\Sigma_2^+ \coloneqq \{0,1\}^\NN$ and ``funnels'' as in~\eqref{eqn:funnel}.

\begin{exercise}
Determine the correspondence between the above definition of a basis and the definition given in \S 1.2 of~\cite{vR67}.
\end{exercise}

Since each ``funnel'' corresponds to some element of $\TTT$ which is either a singleton or empty, we have associated to each Borel subset of $\Sigma_2^+$ an element of $\TTT$, and so $\mu$ yields a measure on $\Sigma_2^+$. Thus we have a notion of ``almost all funnels'' -- we say that the basis is \emph{complete} if almost every funnel contains exactly one point.\footnote{This is not to be confused with the notion of completeness for $\sigma$-algebras.} That is, the set of funnels whose intersection is empty should be measurable, and should have measure zero.  Equivalently, a basis defines a map from $X$ to $\Sigma_2^+$ which takes each point to the ``funnel'' containing it; the basis is complete if the image of this map has full measure.

The existence of a complete basis is the final invariant needed to classify ``nice'' measure spaces.

\begin{theorem}\label{thm:Lebesgue2}
If $(X,\TTT,\mu)$ is separable, non-atomic, and possesses a complete basis, then it is isomorphic to Lebesgue measure on the unit interval.
\end{theorem}
\begin{proof}
Full details can be found in \cite{vR62}; here we describe the main idea, which is that with the completeness  assumption,  the argument from the proof of  Theorem~\ref{thm:Lebesgue} indeed gives    an isomorphism  of measure spaces.  Using the notation from that proof, every infinite intersection $\bigcap_{n\geq 1} A_{x_1\dots x_n}$ corresponds to a  point in the interval, namely
\begin{equation}\label{eqn:limsum}
\lim_{n\to\infty} \sum_{v \prec x_1\cdots x_n} \mu(A_v) = \sum_{\stackrel{n\geq 1}{x_n =1}} \mu(A_{x_1 \cdots x_{n-1} 0}).
\end{equation}
With the exception of a 
countable set (the endpoints of basic intervals of various ranks), every point is the image of at most one ``funnel''.  Completeness guarantees that the is correspondence is indeed   a bijection between sets of full measure. Measurability follows from the fact that   images of  the sets from a basis  are finite unions of intervals.
\end{proof}

In fact, all the measure spaces of interest to us are separable and complete, as the following series of exercises shows.\foot{Moved these exercises together and placed them after the theorem.  Does the sequence of exercises as presented here provide sufficient guidance?}

\begin{exercise}\label{ex:uncountablesum}
Let $X$ be a metric space and fix $x\in X$.  Given $r>0$, let $S_r = \partial B(x,r)$ be the boundary of the ball of radius $r$ centred at $x$.  Show that for any probability measure $\nu$ and any fixed $x$, at most countably many of the $S_r$ have positive measure.
\end{exercise}

\begin{exercise}\label{ex:basis}
Let $X$ be a separable metric space, $\TTT$ the $\sigma$-algebra of Borel sets, and $\nu$ a probability measure.  Use Exercise~\ref{ex:uncountablesum} to show that $(X,\TTT,\nu)$ has a basis $\{\xi_n\}_{n\geq 1}$ such that all boundaries have zero measure -- that is, $\nu(C)=0$ for all $C\in \xi_n$ and $n\geq 1$.
\end{exercise}

\begin{exercise}\label{ex:sepcompmet}
Let $(X,\TTT,\nu)$ be as in Exercise~\ref{ex:basis}, and suppose that in addition $X$ is complete (as a metric space).  Show that the basis constructed in Exercise~\ref{ex:basis} is complete (as a basis).
\end{exercise}

\begin{definition}\label{defLebesgue}
A separable measure space $(X,\TTT,\mu)$ with a complete basis is called a \emph{Lebesgue space.}
\end{definition}

In light of Definition~\ref{defLebesgue}, we can rephrase Exercise~\ref{ex:sepcompmet} as the result that every separable complete metric space equipped with a Borel probability measure is a Lebesgue space.  By Theorem~\ref{thm:Lebesgue2}, every Lebesgue space is isomorphic to the union of unit interval with at most countably many atoms.

It is also worth noting that any separable measure space admits a completion, just as is the case for metric spaces.  The procedure is quite simple; take a basis for $X$ which is not complete, and add to $X$ one point corresponding to each empty ``funnel''.  Thus we need not concern ourselves with non-complete spaces.

\begin{exercise}
Show that every separable measure space which is \emph{not} complete is isomorphic to a set of outer measure one in a Lebesgue space.
\end{exercise}

Thus we have elucidated the promised difference  between the language of sets and that of points: separability is sufficient for the first to  lead to the standard model, while for the second, completeness is also needed.  This  distinction is important theoretically; in particular, it allows us to  separate results which hold  for arbitrary separable measure spaces  (such as most ergodic theorems) from those which hold in Lebesgue spaces (such as von Neumann's isomorphism theorem for dynamical systems with pure point spectrum).

However, non-Lebesgue measure spaces are at least as pathological  for  ``normal mathematics'' as non-measurable sets  or  sets of cardinality higher than continuum. In particular, as a consequence of Exercises~\ref{ex:uncountablesum}--\ref{ex:sepcompmet},  the measure spaces which arise in conjuction with dynamics are all   Lebesgue spaces, so from now on we will restrict our attention to those. 

\subsection{Partitions and $\sigma$-algebras}

We have already seen a simple procedure for associating to each partition $\xi$ of $X$ a sub-$\sigma$-algebra $\BBB(\xi)$ of $\TTT$, and it is natural to ask whether there is a natural class of partitions on which  the morphism $\BBB$ is one-to-one, so that it can be inverted. The answer turns out to be positive if  one considers equivalence classes of partitions  mod zero. 

\begin{definition}
Two partitions $\xi$, $\eta$ of $X$ are \emph{equivalent mod zero} if there exists a set $E \subset X$ of full measure such that
\[
\{C \cap E \mid C\in \xi\} = \{ D \cap E \mid D \in\eta\},
\]
in which case we write $\xi \circeq \eta$.
\end{definition}

As with sets and $\sigma$-algebras, we will always consider partitions up to equivalence mod zero, and will again write $=$ in place of $\circeq$.

\begin{theorem}\label{thm:algtopart}
Given a separable measure space $(X,\TTT,\mu)$ and a sub-$\sigma$-algebra $\AAA\subset \TTT$, there exists a partition $\xi$ of $X$ into measurable sets such that $\AAA$ and $\BBB(\xi)$ are equivalent mod zero.
\end{theorem}
\begin{proof}
Without loss of generality, we may assume that $\mu$ is non-atomic (if $\AAA$ contains any atoms, these can be taken as elements of $\xi$, and there can only be countably many disjoint atoms).  Since $(X,\TTT,\mu)$ is separable, so is $(X,\AAA,\mu)$.  (The metric space $(\AAA,d_\mu)$ is a subspace of $(\TTT,d_\mu)$.)  In particular, we may take a basis $\{\xi_n\}_{n\in\NN}$ for $\AAA$ and define $\xi$ by~\eqref{eqn:xifromxin}.  It remains only to show that $\AAA = \BBB(\xi)$, which we leave as an exercise.
\end{proof}

\begin{exercise}
Complete the proof of Theorem~\ref{thm:algtopart}.\foot{Restored the hypothesis of separability to the theorem and added this exercise.  Sufficient?  Or should we include the solution of the exercise in the proof?}
\end{exercise}


We will denote the partition constructed in Theorem~\ref{thm:algtopart} by $\Xi(\AAA)$.  Analogously to the proofs of Theorems~\ref{thm:Lebesgue} and~\ref{thm:Lebesgue2}, the elements of $\Xi(\AAA)$ may be described explicitly as follows:  without loss of generality, assume that $A_n^0$ and $A_n^1 = X \setminus A_n^0$ are such that each $\xi_n$ has the form $\bigvee_{k=1}^n \{A_k^0, A_k^1\}$, and given $w\in \{0,1\}^\NN$, let $A_w = \bigcap_{n\in \NN} A_n^{w_n}$.  Note that unlike in those proofs, the intersection may contain more than one point -- indeed, some intersections $A_w$ \emph{must} contain more than one point unless $\Xi(\AAA)=\eps$.


\subsection{Measurable partitions}

We now have a natural way to go from a partition $\xi$ to a $\sigma$-algebra $\BBB(\xi) \subset \TTT$, and from a $\sigma$-algebra $\AAA \subset \TTT$ to a partition $\Xi(\AAA)$.

The definition of $\Xi(\cdot)$ in Theorem~\ref{thm:algtopart} guarantees that it is a one-sided inverse to $\BBB(\cdot)$, in the sense that $\BBB(\Xi(\AAA)) = \AAA$ for any $\sigma$-algebra $\AAA$ (up to equivalence mod zero).  So we may ask if the same holds for partitions; is it true that $\xi$ and $\Xi(\BBB(\xi))$
are equivalent in some sense?

We see that since each set in $\Xi(\BBB(\xi))$ is measurable, we should at least demand that $\xi$ not contain any non-measurable sets.  For example, consider the partition $\xi=\{A,B\}$, where $A\cap B=\emptyset$, $A\cup B=X$: then if $A$ is measurable (and hence $B$ as well), we have
\[
\BBB(\xi) = \{\emptyset, A, B, X\}
\]
and $\Xi(\BBB(\xi)) = \{A,B\}$, while if $A$ is non-measurable, we have
\[
\BBB(\xi) = \{\emptyset, X\}
\]
and so $\Xi(\BBB(\xi)) = \nu$.  Thus a ``good'' partition should only contain measurable sets; it turns out, however, that this is not sufficient, and that there are examples where $\Xi(\BBB(\xi))$ is not equivalent mod zero to $\xi$, even though every set in $\xi$ is measurable.

\begin{example}\label{eg:nonmeas}
Consider the torus $\TT^2$ with Lebesgue measure $\lambda$, and let $\xi$ be the partition into orbits of a linear flow $\phi_t$ with irrational slope $\alpha$; that is, $\phi_t(x,y) = (x+t, y+t\alpha)$.  In order to determine $\BBB(\xi)$, we must determine which measurable sets are unions of orbits of $\phi_t$; that is, which measurable sets are invariant.  Because this flow is ergodic with respect to $\lambda$, any such set must have measure $0$ or $1$, and so up to sets of measure zero, $\BBB(\xi)$ is the trivial $\sigma$-algebra!  It follows that $\Xi(\BBB(\xi))$ is the trivial partition $\nu = \{ \TT^2 \}$.

A discrete-time version of this is the partition of the circle into orbits of an irrational rotation.
\end{example}

\begin{definition}\label{def:meas-hull}
The partition $\Xi(\BBB(\xi))$ is known as the \emph{measurable hull} of $\xi$, and will be denoted by $\HHH(\xi)$.  If $\xi$ is equivalent mod zero to its measurable hull, we say that it is a \emph{measurable partition}.

In particular (foreshadowing the next section), if we denote the partition into orbits of some dynamical system by $\OOO$, then $\HHH(\OOO)$ is also known as the \emph{ergodic decomposition} of that system, and is denoted by $\EEE$.\footnote{It should be noted that because we have not yet talked about conditional measures, one may rightly ask just what about this decomposition is ergodic.}
\end{definition}

It is obvious that in general, the measurable hull of $\xi$ is a coarsening of $\xi$; the definition says that if $\xi$ is non-measurable, this is a proper coarsening.\footnote{Compare this with the action of the Legendre transform on functions -- taking the double Legendre transform of any function returns its \emph{convex hull}, which lies on or below the original function, with equality if and only if the original function was convex.}

\begin{exercise}\label{ex:hull}
Show that the measurable hull $\HHH(\xi)$ is the finest measurable partition which coarsens $\xi$---in particular, if $\eta$ is any partition with
\[
\xi \leq \eta < \HHH(\xi),
\]
then $\HHH(\eta) = \HHH(\xi)$, and hence $\eta$ is non-measurable.
\end{exercise}

$\BBB$ gives a map from the class of all partitions to the class of all $\sigma$-algebras, and $\Xi$ gives a map in the opposite direction, which is the one-sided inverse of $\BBB$.  We see that the set of \emph{measurable} partitions is just the image of the map $\Xi$, on which $\HHH$ acts as the identity, and on which $\BBB$ and $\Xi$ are two-sided inverses.

Thus we have a correspondence between measurable partitions and $\sigma$-algebras -- one may easily verify that the operations $\vee$ and $\wedge$ on measurable partitions correspond directly to the operations $\cup$ and $\cap$ on $\sigma$-algebras, and that the relations $\leq$ and $\geq$ correspond directly to the relations $\subset$ and $\supset$.

Example~\ref{eg:nonmeas} shows that the orbit partition for an irrational toral flow is non-measurable; in fact, this is true for \emph{any} ergodic system with more than one orbit, since in this case $\BBB(\OOO)$ is the trivial $\sigma$-algebra, whence $\EEE = \{X\}$ is the trivial partition and $\OOO \neq \EEE = \HHH(\OOO)$.  This sort of phenomenon is widespread in dynamical systems -- for example, we will see in \S3 that in the context of smooth dynamics, the partition into unstable manifolds is non-measurable whenever entropy is positive.


An alternate characterisation of measurability may be motivated by recalling that in the ``toy'' example of a partition into two subsets, the corresponding $\sigma$-algebra had four elements in the measurable case, and only two in the non-measurable case.  In some sense, measurability of the partition corresponds to increased ``richness'' in the associated $\sigma$-algebra.  This is made precise as follows:\footnote{In \cite[p.\ 4]{vR67}, the property described in Theorem~\ref{thm:sigmeas} is given as the definition of measurable.  The result here shows that the two definitions are equivalent.}

\begin{theorem}\label{thm:sigmeas}
Let $\xi$ be a partition of a Lebesgue space $(X,\TTT,\mu)$.  $\xi$ is measurable if and only if there exists a countable set $\{A_n\}_{n\in\NN} \subset \BBB(\xi)$ such that for almost every pair $C_1, C_2 \in \xi$, we can find some $A_n$ which separates them in the sense that $C_1 \subset A_n$, $C_2 \subset X \setminus A_n$.
\end{theorem}
\begin{proof}[Sketch of proof]
The key observation is the fact that such a set $\{A_n\}_{n\in\NN}$ corresponds to a refining sequence of
partitions~\eqref{eqn:parts} defined by
\begin{align*}
\eta_k &= \{ A_k, X\setminus A_k \}, \qquad\quad
\xi_n = \bigvee_{k=1}^n \eta_k. \qedhere
\end{align*}
\end{proof}

\begin{exercise}
Complete the proof of Theorem~\ref{thm:sigmeas}.
\end{exercise}

It may not immediately be clear what is meant by ``almost every pair'' in the statement of Theorem~\ref{thm:sigmeas}.  Recall that the natural projection $\pi\colon X \to \xi$ takes $x \in X$ to the unique partition element $C\in \xi$ containing $x$.\footnote{A word on notation.  There is a natural correspondence between partitions and equivalence relations; if we use $\xi$ to denote the partition, then $\pi$ takes values in $\xi$, whereas if we use $\xi$ to denote the equivalence relation, then $\pi$ takes values in the quotient space $X/\xi$.}  Thus $\xi$, which may be thought of as the space of equivalence classes, carries a measure $\mu_\xi$ which is the pushforward of $\mu$ under $\pi$ -- given a measurable set $E \subset \xi$, we have
\[
\mu_\xi(E) = \mu(\pi^{-1}(E)).
\]
This gives a meaning to the notion of ``almost every'' partition element, and hence to ``almost every pair'' of partition elements. Another way to parse the statement is to see that we may remove some set $E$ of zero measure from $X$ and pass to the ``trimmed-down'' partition $\xi|_{X\setminus E}$, for which the statement holds for every $C_1, C_2$.

Aside from finite or countable partitions into measurable sets (which are obviously measurable), a good example of a measurable partition is given by Example~\ref{eg:lines}, in which the square $[0,1]\times[0,1]$ with Lebesgue measure $\lambda$ is partitioned into vertical lines.  In fact, this is in some sense the \emph{only} measurable partition, just as $[0,1]$ is, up to isomorphism, the only Lebesgue space -- the following result states that a measurable partition can be decomposed into a ``discrete'' part, where each element has positive measure, and a ``continuous'' part, which is isomorphic to the partition of the square into lines.

\begin{theorem}\label{thm:measpart}
Given a measurable partition $\xi$ of a Lebesgue space $(X,\TTT,\mu)$, there exists a set $E \subset X$ such that
\begin{enumerate}
\item Each element of $\xi|_E$ has positive measure (and hence there are at most countably many such elements).
\item $\xi|_{X\setminus E}$ is isomorphic to the partition of the unit square with Lebesgue measure into vertical lines given in Example~\ref{eg:lines}.
\end{enumerate}
\end{theorem}
\begin{proof}
We give a complete proof modulo a technical lemma (Lemma \ref{lem:phconv}), whose proof we only sketch.  Let $E$ be the union of the elements of $\xi$ that have positive measure.  To prove the theorem it suffices to restrict our attention to $X\setminus E$, and so from now on we assume that $E$ is empty and all elements of $\xi$ have measure $0$.

The proof is a more sophisticated version of the argument in the proofs of Theorems~\ref{thm:Lebesgue} and~\ref{thm:Lebesgue2}.  Let $A_n\in \BBB(\xi)$ be as in Theorem~\ref{thm:sigmeas}, and as in the proof of Theorem~\ref{thm:Lebesgue}, write $A_{x_1\cdots x_n} = A_1^{x_1} \cap \cdots \cap A_n^{x_n}$, where $A_n^0 = A_n$ and $A_n^1 = X\setminus A_n$.  The idea is that mimicking the proof of Theorem~\ref{thm:Lebesgue2}, we will construct an isomorphism $\rho$ that sends $A_w$ to the vertical strip in $[0,1]\times [0,1]$ whose horizontal footprint is the interval with length $\mu(A_w)$ and left endpoint at $\sum_{v\prec w} \mu(A_v)$.  What remains is to describe the vertical coordinate of the isomorphism.

Before doing this, first observe that the previous paragraph defines a map $\pi\colon X\to [0,1]$ such that if $x\in \{0,1\}^\NN$ and $C\in \xi$ are such that $C = \bigcap_{n\geq 1} A_{x_1\cdots x_n}$, then $\pi(z) = \sum_{x_n=1} \mu(A_{x_1\cdots x_{n-1} 0})$ for every $z\in C$.  Since almost every $C$ admits such an $x$, we see that $\xi$ is equivalent mod zero to
\begin{equation}\label{eqn:preimages}
\pi^{-1}(\eps_{[0,1]}) = \{ \pi^{-1}(a) \mid a\in [0,1] \},
\end{equation}
the \emph{partition into preimages} for the map $\pi$.

So far we have associated to almost every point $z\in X$ a sequence $x = x(z) \in \{0,1\}^\NN$ that determines in which element of $\xi$ the point $z$ lies.  Now fix another sequence of sets $B_n\in \TTT$, this time requiring that they generate the entire $\sigma$-algebra: $\bigvee_{n\geq 1} \{B_n^0, B_n^1\} = \eps$, the partition into points, where $B_n^0 = B_n$ and $B_n^1 = X\setminus B_n$.  As before, given $w\in \{0,1\}^n$, let $B_w = \bigcap_{1\leq k\leq n} B_k^{w_k}$.  Thus to almost every $z\in X$ we can also associate $y=y(z)\in \{0,1\}^\NN$ with the property that $z\in \bigcap_{n\geq 1} B_{y_1\cdots y_n}$.

The map $\pi$ only depends on $x(z)$, and so abusing notation slightly, we will find a function $\phi(x,y)$ such that the map $\rho\colon X\to [0,1]\times [0,1]$ given by $\rho(z) = (\pi(x), \phi(x,y))$ is the desired isomorphism.  Just as $\pi(x)$ is a sum of measures of partition elements, so too $\phi(x,y)$ will be a sum of \emph{conditional} measures of partition elements.

Given a word $w\in \{0,1\}^n$, we define functions $\ph_m^w\colon \{0,1\}^\NN \to [0,1]$ by
\[
\ph_m^w(x) = \frac{\mu(B_w \cap A_{x_1\cdots x_m})}{\mu(A_{x_1\cdots x_m})}.
\]
Thus $\ph_m^w(x)$ is the conditional measure of $B_w$ within the partition element $A_{x_1\cdots x_m}$.  Figure~\ref{fig:lebesgue2} illustrates the procedure for varying $m$ and fixed $w$; observe that the total shaded area under the function remains constant within each $\rho(A_v)$ as $m$ increases.

\begin{figure}[tbp]
	\includegraphics{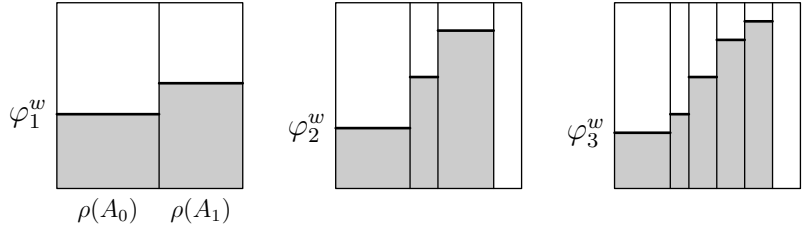}
	\caption{Obtaining $\ph^w$ as a limit of $\ph_m^w$.}
	\label{fig:lebesgue2}
\end{figure}

\begin{lemma}\label{lem:phconv}
There exists a measurable function $\ph^w\colon \{0,1\}^\NN\to [0,1]$ such that $\ph_m^w\to \ph^w$ almost everywhere.
\end{lemma}
\begin{proof}[Sketch of proof]
Full details are in~\cite[Lemma 4]{Via}.  The idea is to show that for every $\alpha<\beta$, the set
\[
S_{\alpha,\beta} = \{x \mid \llim \ph_m^w(x) < \alpha < \beta < \ulim \ph_m^w(x) \}
\]
has zero measure, since the set of points without convergence is a countable union of such sets.  To show this, one observes that for every $x\in S_{\alpha,\beta}$ there exist $c_1 < d_1 < c_2 < d_2 < \cdots$ such that $\ph_{c_i}^w < \alpha < \beta < \ph_{d_i}^w$ for all $i$.  Let
\[
C_i = \bigcup_{x\in S_{\alpha,\beta}} A_{x_1\cdots x_{c_i(x)}}, \qquad
D_i = \bigcup_{x\in S_{\alpha,\beta}} A_{x_1\cdots x_{d_i(x)}},
\]
so that since $S_{\alpha,\beta} \subset C_{i+1} \subset D_i \subset C_i$ for all $i$, we have
\[
\alpha \mu(C_i) > \mu(B_w \cap C_i) > \mu(B_w \cap D_i) > \beta \mu(D_i).
\]
Writing $S' = \bigcap C_i = \bigcap D_i$, we have $\alpha \mu(S') \geq \beta \mu(S')$, so $\mu(S')=0$.
\end{proof}

\begin{figure}[tbp]
	\includegraphics{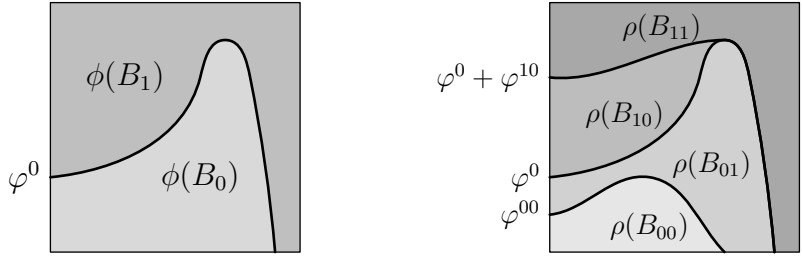}
	\caption{Defining $\rho$ using the functions $\ph^w$.}
	\label{fig:lebesgue3}
\end{figure}

The function $\ph^w(x)$ may be interpreted as the conditional measure of the set $B_w$ in the partition element $\xi(x)$, a point which we elaborate on later.  For the moment we conclude the proof of Theorem~\ref{thm:measpart} by putting
\[
\phi(x,y) = \sum_{\{ n \mid y_n=1\}} \ph^{y_1\cdots y_{n-1} 0}(x),
\]
which takes the place of \eqref{eqn:limsum}.  Now $\rho(z) = (\pi(x(z)), \phi(x(z),y(z)))$ is the desired isomorphism.   Figure~\ref{fig:lebesgue3} illustrates the first two steps in the definition of $\rho$, although we point out that the functions $\ph^w$ need only be measurable, not smooth as in the picture.
\end{proof}


In the course of the previous proof, we showed that a measurable partition $\xi$ can be found as the partition into preimages~\eqref{eqn:preimages} associated to a certain map.  The following result, whose proof (which uses Theorem~\ref{thm:sigmeas}) is left as an exercise, gives conditions under which the converse is true, and thus gives another criterion that can be used to establish measurability of a partition (a further criterion is found in Exercise~\ref{ex:measaslimit}).


\begin{theorem}\label{thm:preimage-partition}
Let $X$ be a complete metric space, $\mu$ a Borel measure on $X$, $Y$ a second countable topological space, and $\pi\colon X\to Y$ a Borel map (that is, preimages of Borel sets 
are Borel). Then the partition into preimages defined by~\eqref{eqn:preimages} is measurable.
\end{theorem}

\begin{example}\label{eg:peacock}
Let $C \subset [0,1]$ be the usual middle-third Cantor set, which has Lebesgue measure $0$ but contains uncountably many points.  Then there is a bijection from $C$ to $[0,1] \setminus C$, and so we may take a partition $\xi$ of $[0,1]$ such that each element of $\xi$ contains exactly two points, one in $C$ and one not in $C$.  Using the characterisation in Theorem~\ref{thm:sigmeas}, we see that $\xi$ is measurable, since we may take for our countable collection the set of intervals with rational endpoints.  Further, this partition is equivalent mod zero to the partition into points.
\end{example}

The situation described in Example~\ref{eg:peacock}, where a partition is in some sense finer than it appears to be, happens all the time in ergodic theory.  A fundamental example is the so-called \emph{Fubini's nightmare}, in which a partition which seems to divide the space into curves in fact admits a set of full
measure intersecting each partition element exactly once, and hence is equivalent mod zero to the partition into points (we will return to this example in~\S\ref{sec:folmeas}).

This sort of behaviour stands in stark contrast to \emph{absolute continuity} -- but in order to make any sense of that notion, we must first discuss conditional measures.

\subsection{Conditional measures on measurable partitions}

If a partition element $C$ carries positive measure (which can only be true of countably many elements), then we can define a conditional measure on $C$ by the obvious method; given $E \subset C$, the conditional measure of $E$ is
\begin{equation}
\mu_C(E) \coloneqq \frac{\mu(E)}{\mu(C)}.
\end{equation}
However, for many partitions arising in the study of dynamical systems, such as the partitions into stable and unstable manifolds which will be discussed later, we would also like to be able to define a conditional measure on partition elements of zero measure, and to do so in a way which allows us to reconstruct the original measure.

The model to keep in mind is the canonical example of a measurable partition, the square partitioned into vertical lines (Example~\ref{eg:lines}). Then denoting by $\lambda$, $\lambda_1$, and $\lambda_2$ the Lebesgue measures on the square, the horizontal unit interval, and vertical intervals, respectively, Fubini's theorem says that for any integrable $f\colon [0,1]^2 \to \RR$ we have
\begin{equation}\label{eqn:Fubini}
\int_{[0,1]^2} f(x,y) \,d\lambda(x,y) = \int_{[0,1]} \int_{[0,1]} f(x,y) \,d\lambda_2(y) \,d\lambda_1(x).
\end{equation}

By Theorem~\ref{thm:measpart}, any measurable partition of a Lebesgue space is isomorphic to the standard example -- perhaps with a few elements of positive measure hanging about, but these will not cause any trouble, as we already know how to define conditional measures on them. Taking the pullback of the Lebesgue measures $\lambda_1$ and $\lambda_2$ under this isomorphism, we obtain a \emph{factor measure} $\mu_\xi$ on $\xi$, which corresponds to the horizontal unit interval (the set of partition elements), and a family of \emph{conditional measures} $\{\mu_C\}_{C\in\xi}$, which correspond to the vertical unit intervals.

Note that the factor measure is exactly the measure on the space of partition elements which was described in the last section.  Note also that although the measure $\lambda_2$ was the same for each vertical line (up to a horizontal translation), we can make no such statement about the measures $\mu_C$, as the geometry is lost in the purely measure theoretic isomorphism between $X$ and $[0,1]^2$.  The key property of these measures is that for any integrable function $f\colon X \to \RR$, the function
\begin{equation}\label{eqn:intpart}
\begin{aligned}
\xi &\to \RR, \\
C &\mapsto \int_C f \,d\mu_C
\end{aligned}
\end{equation}
is measurable, and we have
\begin{equation}\label{eqn:condfubini}
\int_X f\,d\mu = \int_{\xi} \int_C f\,d\mu_C \,d\mu_\xi.
\end{equation}

Each $\mu_C$ is ``supported'' on $C$ in the sense that $\mu_C(C)=1$, but the reader is cautioned that the measure theoretic support of a measure (which is not uniquely defined) is a different beast than the topological support of a measure, and that $\supp \mu_C$ may not be equal to $C \cap \supp \mu$, as the following example shows.

\begin{example}
Let $A\subset [0,1]$ be such that both it and its complement $A^c = [0,1] \setminus A$ intersect every interval in a set of positive
measure.\footnote{Such a set can be constructed, for instance, by repeatedly removing and replacing appropriate Cantor sets of positive measure.}  Let $\lambda_1$ be one-dimensional Lebesgue measure, and define a measure $\mu$ on the unit square by
\[
\mu(E) = \lambda_1(E\cap (A\times\{0\})) + \lambda_1( E\cap
(A^c \times \{1\}))
\]
for each $E\subset [0,1]\times [0,1]$.  Then the topological support of $\mu$ is the union of two horizontal lines, $[0,1]\times\{0,1\}$, and intersects each partition element in two points, but the conditional measures are $\delta$-measures
supported on a single point.
\end{example}

We cannot in general write a simple formula for the conditional measures, as we could in the case where partition elements carried positive weight, so on what grounds do we say that these conditional measures exist?  The justification above relies on the characterisation of measurable partitions given by Theorem~\ref{thm:measpart}.  Related proofs that do not require constructing an isomorphism to the square are presented in Viana's notes~\cite{Via} (which draw on Rokhlin's paper~\cite{vR62}) and in Furstenberg's book~\cite{hF80}.  These use methods from functional analysis, principally the Riesz representation theorem, made available by defining a topology on $X$.

\subsection{Measure classes and absolute continuity}

Let $(X,\TTT)$ be a measurable space, and consider the set $\MMM$ of all measures on $X$.  This set has various internal structures which may be of importance to us; for the time being, we focus our attention on the fact, guaranteed by the Radon--Nikodym Theorem, that measures come in classes.  This theorem addresses the relationship between two measures $\nu$ and $\mu$, and allows us to pass from a qualitative statement to a quantitative one; namely, if $\nu$ is absolutely continuous with respect to $\mu$,\footnote{This means that if $\mu(E)=0$, then $\nu(E)=0$ as well, a state of affairs which is denoted $\nu \ll \mu$.} then there exists a measurable function $d\nu / d\mu$, known as the \emph{Radon--Nikodym derivative}, which has the property that
\[
\nu(E) = \int_E \frac{d\nu}{d\mu}(x) \,d\mu(x)
\]
for any $E \in \TTT$.\footnote{As an aside, note that if we change the $\sigma$-algebra $\TTT$, then we also change the Radon--Nikodym derivative, since $d\nu / d\mu$ must be measurable with respect to $\TTT$.  This fact is crucial to the proof of the Birkhoff Ergodic Theorem in~\cite{KH95}.}

Given a reference measure $\mu$ and any other measure $\nu$, we also have the \emph{Radon--Nikodym decomposition} of $\nu$; that is, we may write $\nu = \nu_1 + \nu_2$, where $\nu_1 \ll \mu$ and $\nu_2 \perp \mu$ (the latter means that there exists $A \in \TTT$ such that $\nu_2(A) = 1$ and $\mu(A) = 0$).

The notion of absolute continuity plays an important role in smooth dynamics, where we have a reference measure class given by the smooth structure of the manifold in question, and are often particularly interested in measures which are absolutely continuous with respect to this measure class.

Given a partition $\xi$, we may also speak of $\nu$ as being absolutely continuous with respect to $\mu$ on the elements of $\xi$ by passing to the conditional measures $\nu_C$ and $\mu_C$ and applying the above definitions.  For example, if we fix $x\in [0,1]$ and write $\delta_x$ for the measure on $[0,1]$ with
\[
\delta_x(E) = \begin{cases}
0 & x\notin E, \\
1 & x\in E,
\end{cases}
\]
and $\lambda$ for Lebesgue measure on $[0,1]$, then the product measure $\delta_x \times \lambda$ fails to be absolutely continuous with respect to Lebesgue measure on $[0,1]^2$, but \emph{is} absolutely continuous on the elements of the partition into vertical lines.  This weaker version of absolute continuity is an important notion in smooth dynamics, where it allows us to ask not just if a measure is absolutely continuous on the manifold as a whole, but if it is absolutely continuous in certain directions, which correspond to the various rates of expansion and contraction given by the Lyapunov exponents.  In particular, we are often interested in measures which are absolutely continuous on unstable leaves, so-called \emph{SRB measures}.

\begin{example}
Let $C\subset [0,1]$ be the usual middle-thirds Cantor set, and let $\mu$ be the probability measure on $C$ that gives weight $2^{-n}$ to each of the basic intervals at the $n$th stage of the construction (which have length $3^{-n}$).  Then the product measure $\mu\times\lambda$ is not absolutely continuous with respect to Lebesgue measure on the square, but is absolutely continuous with respect to the partition into vertical lines.
\end{example}

\section{Measurable Dynamics}

\subsection{Partitions of times past and future}

Now we consider not just a set $X$, but a dynamical system $(X,T)$, where $T\colon X\to X$ is a map whose iterates $T^n$ are of interest.  Generally speaking $X$ carries some structure -- topological, measure-theoretic, metric, manifold -- which is preserved by the action of $T$.  One of the key notions in dynamics is that of \emph{invariance}: the map $T$ sends points to points, sets to sets, measures to measures, functions to functions, and we are interested in properties and characterisations of points, sets, measures, functions which are invariant under the action of $T$.

In this section we will assume that $(X,\TTT,\mu)$ is a measure space as in the previous  section, and that $T\colon X\to X$ is 
a measure-preserving transformation -- that is, that $\mu(E) = \mu(T^{-1}E)$ for all $E\in \TTT$.  We often express the equality $\mu=\mu\circ T^{-1}$ by saying that the measure $\mu$ is invariant under the action of $T$.  When $X$ is a metric space or a manifold and $T$ is a continuous or smooth map, it is often (but not always) the case that there are very many invariant measures.  For the time being, though, we will only consider a single invariant measure.\footnote{Many of the definitions and results here work for any measure-preserving transformation $T$, but some also require $T$ to be invertible with measure-preserving inverse.  In this case we also have $\mu(T(E))= \mu(E)$ for every $E\in \TTT$, which is not necessarily the case for non-invertible transformations.}

We may consider the property of invariance for partitions as well; we say that a partition $\xi$ is invariant if $T^{-1}(C) \in \xi$ for every $C \in \xi$, that is, if the preimage of a  partition element is again a single partition element.\footnote{In~\cite{vR67}, such a partition is said to be \emph{completely invariant}, and \emph{invariant} instead refers to the weaker property that $T^{-1}\xi \leq \xi$, so that the preimage of a partition element is a union of partition elements.}  This is written as $T^{-1}\xi = \xi$, where
\[
T^{-1}\xi = \{\, T^{-1}(C) \mid C \in \xi \,\}.
\]

Given an invariant partition $\xi$, let $\pi$ denote the canonical projection $X \to \xi$, as before.  Then $T$ induces an action $\pi \circ T \circ \pi^{-1}$ on the space of partition elements $X/\xi$, and the dynamics of $T$ may be viewed as a skew product over this action.

In light of the correspondence between measurable partitions and $\sigma$-algebras discussed in the previous section, we may also consider invariant $\sigma$-algebras, those for which $T^{-1}\AAA = \AAA$.  It is then reasonable to ask if there is a natural way to associate to an arbitrary partition or $\sigma$-algebra one which is invariant.  One obvious way is to
take a $\sigma$-algebra $\AAA$, and consider the sub-$\sigma$-algebra $\AAA' \subset \AAA$ which contains all the $T$-invariant sets in $\AAA$.\footnote{Of course, $\AAA$ may not contain \emph{any} non-trivial $T$-invariant sets.} However, there is another important construction, which we now examine.

Now we assume that $T$ is an automorphism, i.e., invertible with measure-preserving inverse.  Let $\xi$ be a finite partition of $X$ into measurable sets, and define
\[
\xi_T \coloneqq \bigvee_{n\in\ZZ} T^n \xi = \lim_{n\to\infty} \bigvee_{j=-n}^n T^j\xi.
\]
The elements of this partition are given by $\bigcap_{n\in\ZZ} T^n C_n$, where $C_n \in \xi$.  Observe that $x\in T^n C_n$ if and only if $T^{-n}(x) \in C_n$, and so knowing which element of $T^n\xi$ the point $x$ lies in corresponds to knowing in which element of $\xi$ the points $T^{-n}(x)$ lies.  This is commonly referred to as the \emph{coding} of the trajectory of $x$: knowing which element of $\xi_T$ the point $x$ lies in is equivalent to knowing the coding of the entire trajectory of $x$, both forward and backward.

Because each of the partitions $\xi^{(n)} \coloneqq \bigvee_{j=-n}^n T^j\xi$ have finitely many elements, all measurable, these partitions themselves are measurable, and we have $\xi^{(n)} = \Xi(\BBB(\xi^{(n)}))$; passing to the limit, we see that $\xi_T = \Xi(\BBB(\xi_T))$, so $\xi_T$ is measurable as well.

\begin{exercise}\label{ex:measaslimit}
Show that a partition $\xi$ is measurable if and only if it is the limit ($\xi = \bigvee_{n=1}^\infty \xi_n$) of an increasing sequence $\{\xi_n\}_{n\in\NN}$ of finite partitions into measurable sets.  Indeed, show that $\xi$ is measurable if the $\xi_n$ are \emph{any} measurable partitions.
\end{exercise}

\begin{example}\foot{Removed incorrect claim from previous exercise and added counterexample here.}
Let $X = \{0,1\}^\NN$ and let $\mu$ be the Bernoulli measure that gives weight $2^{-n}$ to each $n$-cylinder.  Let $\xi_n$ be the partition induced by the equivalence relation $x\sim y$ iff $x_k = y_k$ for all $k\geq n$, and let $\xi = \bigwedge_{n=1}^\infty \xi_n$.

Let $A$ be the set of all $x$ such that the terms $x_j$ are eventually $0$ (that is, there exists $k$ such that $x_j=0$ for all $j\geq k$).  Then $A\in \BBB(\xi_n)$ for all $n$, so $A\in \BBB(\xi)$; indeed, $A$ is an element of $\xi$.  
But $\BBB(\xi)$ is the trivial $\sigma$-algebra since all elements are shift-invariant and $\mu$ is ergodic.  So $\HHH(\xi)$ is the trivial partition, hence $\xi$ is not measurable.  This shows that the counterpart to Exercise~\ref{ex:measaslimit} for \emph{decreasing} sequences of finite measurable partitions is false.
\end{example}

It follows immediately from the construction of $\xi_T$ that it is an invariant partition, whose $\sigma$-algebra is very different from the invariant $\sigma$-algebra described above.

\begin{example}\label{eg:xiT}
Consider the space of doubly infinite sequences on two symbols,
\[
X=\Sigma_2 = \{0,1\}^\ZZ = \{\,(x_n)_{n\in\ZZ} \mid x_n\in\{0,1\}\,\},
\]
and let $T$ be the shift $\sigma\colon (x_n)_{n\in\ZZ} \mapsto (x_{n+1})_{n\in\ZZ}$.  Equip $X$ with the Bernoulli measure $\mu$ which gives each $n$-cylinder weight $2^{-n}$.

Geometrically, $X$ may be thought of as the direct product of two Cantor sets $C$ (each corresponding to the one-sided shift space $\Sigma_2^+$). In this picture, $T$ acts on each copy of $C\times C\subset [0,1]\times [0,1]$ as follows: draw two rectangles of width $1/3$ and height $1$, each of which contains half of the horizontal Cantor set; contract each rectangle in the vertical direction by a factor of $3$; expand it in the horizontal direction by the same factor; and finally, stack the resulting rectangles one on top of the other, as in Figure~\ref{fig:cantor-shift}.

\begin{figure}[tbp]
	\includegraphics{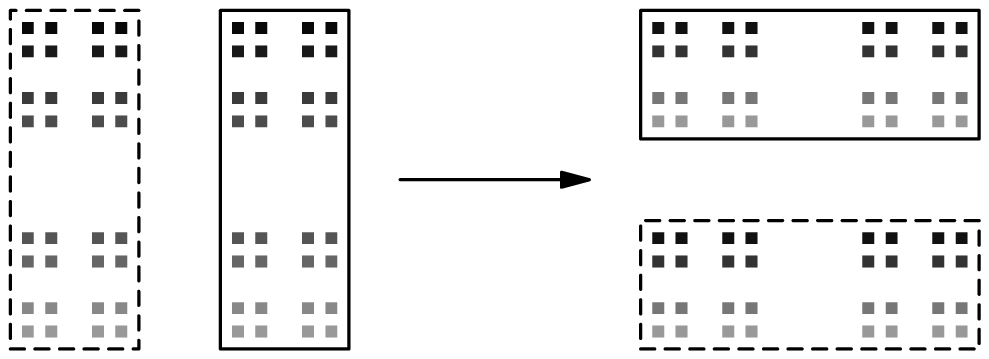}
	\caption{Visualising the action of $\sigma\colon \Sigma_2 \to \Sigma_2$.}
	\label{fig:cantor-shift}
\end{figure}

Now let $\xi$ be the partition of $X$ into one-cylinders; that is, $\xi=\{C_0,C_1\}$, where
\[
C_i \coloneqq \{\, (x_n)_{n\in\ZZ} \in \Sigma_2 \mid x_0=i \,\}.
\]
Each one-cylinder $C_i$ corresponds to one of the two vertical rectangles in the above description, and the reader may verify that in this case, $\xi_T$ is the partition into points.
\end{example}

Another important partition, which is \emph{not} necessarily invariant, is
\[
\xi^- \coloneqq \bigvee_{n=0}^{\infty} T^{-n} \xi.
\]
As discussed above, an element of $\xi_T$ corresponds to trajectories with the same coding for both positive and negative $n$ -- in other words, trajectories with the same past and future relative to the partition $\xi$.  By contrast, $\xi^-$ corresponds to just the infinite future -- points whose forward iterates lie in the same elements of $\xi$ may have backwards iterates lying in different elements of $\xi$.\footnote{In the next section, we will see that for a smooth dynamical system $\xi^-$ can be interpreted as a partition into local stable manifolds.}



This last statement is just another way of saying that $\xi^-$ is not necessarily invariant under the action of $T$.  However, we do have
\[
\xi^- = \xi \vee T^{-1}\xi^- \geq T^{-1}\xi^-;
\]
that is, $\xi^-$ is an \emph{increasing partition}.
\footnote{Note that $\xi^-$ is increasing in the sense that it refines its pre-image; for this to be the case, each individual element must increase in size under $T^{-1}$, and hence \emph{decrease} in size under $T$.  Thus one could also reasonably define increasing partitions as those for which $\xi \leq T^{-1}\xi$, which is
the convention followed in~\cite{LY85a,LY85b}.}


\begin{example}\label{eg:ximinus}
Let $X$, $T$, and $\xi$ be as in Example~\ref{eg:xiT}; then the elements of $\xi^-$ are the sets
\[
C(x) = \{\, y\in\Sigma_2 \mid y_n = x_n\ \forall n\geq 0 \,\},
\]
each of which is a copy of $\Sigma_2^+$, and corresponds in the geometric picture of Figure~\ref{fig:cantor-shift} to a vertical Cantor set
\[
\{t\}\times C \subset C\times C \subset [0,1]\times [0,1],
\]
where $t \in C$ and $C$ is the Cantor set mentioned previously.  Note that each element of the partition $T^{-1}\xi^-$ is a union of two such vertical Cantor sets related by a horizontal translation by $\frac 12$.  Thus we have $T^{-1}\xi^- < \xi^-$.
\end{example}

\begin{example}
Let $X$ be the unit circle and $T$ a rotation by an irrational multiple of $\pi$.  Let $\mu$ be Lebesgue measure and $\xi$ be the partition into two semi-circles.  Then $\xi^-$ and $T^{-1}\xi^-$ are both the partition into points.  In particular, $T^{-1}\xi^- = \xi^-$.
\end{example}

The previous two examples illustrate a dichotomy; either $T\xi^- = \xi^-$, and $\xi^-$ is in fact invariant, or $T\xi^-$ is a proper
refinement of $\xi^-$, which is thus not invariant, as in Example~\ref{eg:ximinus}.  As we will soon see, there are fundamental differences between the two cases.

\subsection{Entropy}

What is the difference between the two cases just discussed, between the case $T\xi^- = \xi^-$ and the case $T\xi^- > \xi^-$?  The key word here is \emph{entropy}; recall that the entropy of a transformation $T$ with respect to a partition $\xi$ is defined as
\begin{equation}\label{eqn:entropy}
h_\mu(T,\xi) \coloneqq \lim_{n\to\infty} \frac{1}{n} H_\mu\left(\bigvee_{k=0}^{n-1} T^{-k} \xi \right),
\end{equation}
where $H_\mu$ is the information content of a (finite or  countable) partition,  given by the following formula (using the convention that $0\log 0 = 0$):
\begin{equation}\label{eqn:information}
H_\mu(\eta) \coloneqq -\sum_{C\in \eta} \mu(C)\log\mu(C).
\end{equation}
This may be interpreted as the expected amount of information that we gain if we know which element of $\eta$ a point $x$ lies in; similarly, the entropy $h_\mu(T,\xi)$ is the average information we gain per iteration of $T$.

\begin{exercise}
Show that if $\xi^-$ is invariant, then $\xi$ has zero entropy, $h_\mu(T,\xi)=0$, whereas if $\xi^-$ is a proper refinement of $T^{-1}\xi^-$, then the partition carries positive entropy, $h_\mu(T,\xi)>0$.
\end{exercise}

The notion of entropy is intimately connected with one more partition canonically associated with $\xi$, defined as
\begin{equation}\label{eqn:Pixi}
\Pi(\xi) \coloneqq \HHH\left(\bigwedge_{n=1}^\infty T^{-n} \xi^-\right).
\end{equation}
Recall that the intersection $\xi\wedge\eta$ of two partitions is the finest partition which coarsens both $\xi$ and $\eta$; if the partitions are measurable, then this corresponds to taking the intersection $\BBB(\xi) \cap \BBB(\eta)$ of the $\sigma$-algebras.

Observe that $T^{-(n+1)}\xi^- \leq T^{-n}\xi^-$ for every $n$, and so
\begin{equation}\label{eqn:PixiN}
\bigwedge_{n=1}^N T^{-n} \xi^- = T^{-N} \xi^- =: \xi_N^-.
\end{equation}
The partition in~\eqref{eqn:PixiN} corresponds to knowing what happens after time $N$ (relative to the partition $\xi$), but having no information on what happens before then.\foot{Removed offending statements and exercises - check to make sure the discussion is correct now.}  


\begin{exercise}\label{ex:Pixi}
Let $X = \{0,1\}^\ZZ\cup\{2,3\}^\ZZ$ and let 
$T$ be the shift $\sigma\colon (x_n)_{n\in\ZZ} \mapsto (x_{n+1})_{n\in\ZZ}$; this is a simple example of a non-transitive subshift of finite type, and comprises two independent copies of the system in Examples~\ref{eg:xiT} and~\ref{eg:ximinus}.  Equip $X$ with the Bernoulli measure $\mu$ which gives each $n$-cylinder weight $(1/4)^n$, and consider the partition $\xi = \{C_0,C_1,C_2,C_3\}$ into $1$-cylinders.

Show that $\Pi(\xi)$ is the partition $\{ C_0 \cup C_1, C_2 \cup C_3 \}$, which separates $X$ into two copies of $\Sigma_2$.  
Generalise this result to an arbitrary non-transitive subshift of finite type.
\end{exercise}


The three partitions we have constructed from $\xi$ are related as follows:
\begin{equation}\label{eqn:xiTxiPixi}
\xi_T \geq \xi^- \geq \Pi(\xi).
\end{equation}
Using the partitions $\xi_N^-$ from~\eqref{eqn:PixiN}, we see that $\xi^- = \xi^-_0$, while $\xi_T$ and $\Pi(\xi)$ may be thought of as the limits of $\xi^-_N$ as $N$ goes to $-\infty$ and $+\infty$, respectively.  

If $h(T,\xi)=0$, then $\xi^- = \xi^-_N$ for all $N$, and all three partitions $\xi_T$, $\xi^-$, and $\Pi(\xi)$ are equal; there is nothing new under the sun, as it were.  In the positive entropy case, each is a proper refinement of the next, as in Examples~\ref{eg:xiT} and~\ref{eg:ximinus}, and Exercise~\ref{ex:Pixi}.  Although all three are measurable, only $\xi_T$ and $\Pi(\xi)$ are always invariant; $\xi^-$ is not invariant except in the zero entropy case.

\begin{exercise}\label{ex:Pixizero}
Show that the partition $\Pi(\xi)$ derived in Exercise~\ref{ex:Pixi} has zero entropy (and thus $\Pi(\Pi(\xi)) = \Pi(\xi)$, so the operator $\Pi$ is idempotent).
\end{exercise}

The result of Exercise~\ref{ex:Pixizero} is actually quite general, and we would like to somehow think of $\Pi(\xi)$ as the ``zero entropy'' coarsening of $\xi^-$.  Notice that $\Pi(\xi)$ may be a continuous partition, whose elements all have zero measure; in this case, the usual definition of entropy makes no sense, and the information function $H$ must be redefined.  We will return to this point in \S\ref{sec:cond_ent}; for now, we simply state that the appropriate meaning of ``zero entropy'' for the continuous partition $\Pi(\xi)$ is that for any finite partition $\eta\le\Pi(\xi)$ we have $h_{\mu}(T,\eta)=0$.

\begin{theorem}\label{thm:zeroentropy}
Let $\eta \leq \Pi(\xi)$ be a finite or countable partition with finite entropy.  Then $h_\mu(T,\eta)=0$.
\end{theorem}
\begin{proof} See \cite{vR67}.
\end{proof}

\subsection{The Pinsker partition}

We may consider the set of all partitions with the property exhibited by $\Pi(\xi)$ in Theorem~\ref{thm:zeroentropy}. This set has an supremum in the partially ordered set of all partitions;
that is, there exists a partition $\pi(T)$ which is the finest (biggest) partition such that every finite partition coarser (smaller) than it has zero entropy. This is the \emph{Pinsker partition}, and we may rephrase the above statement as the fact that a finite partition $\eta$ has $h_\mu(T,\eta)=0$ if and only if $\eta \leq \pi(T)$.

Equivalently, $\pi(T)$ may be defined through its $\sigma$-algebra; consider all finite or countable measurable
partitions with zero entropy, and take the union of their associated $\sigma$-algebras.  This union is the \emph{Pinsker $\sigma$-algebra}, whose associated measurable partition is $\pi(T)$.

Even more concretely, we have the following criterion:  a set $E\in\TTT$ is contained in the Pinsker $\sigma$-algebra if and only if the partition $\xi=\{E,X\setminus E\}$ has $h_\mu(T,\xi)=0$.  Analogously, the Pinsker partition is the join of all zero-entropy partitions.

The Pinsker partition may be thought of as the canonically defined zero entropy part of a measure preserving transformation; there are two extreme cases.  On the one hand, we may have $\pi(T)=\eps$, the partition into points, in which case \emph{every} finite partition is a coarsening of $\pi(T)$, and hence has zero entropy. Thus $T$ is a zero entropy transformation, $h_\mu(T)=0$.  At the other extreme, we may have $\pi(T)=\nu$, the trivial partition $\{X\}$, in which case every finite partition has positive entropy, and we say that $T$ is a \emph{K-system}.

Upon factoring by the Pinsker partition, we can view an arbitrary measure-preserving transformation as a skew product over its zero entropy part.

\subsection{Conditional entropy}\label{sec:cond_ent}

At this point we must grapple with the difficulty hinted at before Theorem~\ref{thm:zeroentropy}.  That is, we would like to make sense of the notion of entropy of $T$ relative to a partition for as broad a class of partitions as possible.  The definition~\eqref{eqn:entropy} relies on the formula~\eqref{eqn:information} for the information content of a partition (often referred to simply as the entropy of the partition); as we observed earlier, this only makes sense when $\xi$ is a finite or countable partition whose elements carry positive measure.  For a continuous partition, such as the partitions into stable and unstable manifolds which will appear in the next section, or the partition of the unit square into vertical lines which we have already seen, this definition is useless, since $\mu(C)=0$ for each individual
partition element $C$.

The way around this impasse is to recall the definition of conditional entropy, and adapt it to our present situation by making use of a system of conditional measures, which as we have seen may be defined for a measurable partition even when individual elements have measure zero.

To this end, we first observe that if we define the \emph{information function} $I_\mu^\eta(x) = -\log \mu(\pi_\eta(x))$, where $\pi_\eta$ is the canonical projection taking $x$ to the element of $\eta$ in which it is contained, then the definition~\eqref{eqn:information} of $H_\mu(\eta)$ can be replaced by the following formula:
\[
H_\mu(\eta) = \int_X I_\mu^\eta(x) \,d\mu(x).
\]
That is, the entropy of the partition $\eta$ is the expected value of the information function.  Similarly, given two finite or countable partitions $\xi$ and $\eta$, one definition of the conditional entropy $H_\mu(\xi|\eta)$ is as the expected value of the conditional information function
\begin{equation}\label{eqn:condinfo}
I_\mu^{\xi,\eta} \colon x\mapsto -\log\mu_{\pi_\eta(x)}(\pi_\xi(x)).
\end{equation}

The useful feature of~\eqref{eqn:condinfo} is that it works for \emph{any} measurable partitions $\xi$ and $\eta$, including continuous ones -- all we need is a system of conditional measures.

We could also avoid the explicit use of the information function and consider the usual entropy $H_{\mu_C} (\xi|_C)$ on each partition element $C\in\eta$, then integrate using the factor measure to obtain $H_\mu(\xi | \eta)$.  Provided $\xi|_C$ has elements of positive conditional measure $\mu_C$, the usual entropy will be well defined, and we are in business.

\begin{exercise}\label{ex:condentropy}
Let $\xi$ be a finite or countable partition, so that we may apply the usual definition of entropy, and show that
\[
h_\mu(T,\xi) = H_\mu(\xi^- \mid T^{-1}\xi^-).
\]
Further, show that if $\xi$ is increasing ($\xi \geq T^{-1}\xi$), we have
\begin{equation}\label{eqn:condentropy}
h_\mu(T,\xi) = H_\mu(\xi \mid T^{-1}\xi).
\end{equation}
\end{exercise}

Since the right hand side of~\eqref{eqn:condentropy} is defined for any measurable increasing partition, and is shown by Exercise~\ref{ex:condentropy} to agree with the usual definition of entropy for finite and countable partitions, we may take it as a definition of entropy for an arbitrary measurable increasing partition.

The key fact connecting these considerations to smooth dynamics is the observation that if $h_\mu(T)=0$, then the conditional entropy on each partition element is $0$, which in the context of the next section will imply that conditional measures on stable and unstable leaves must be atomic.

\section{Foliations and Measures}\label{sec:folmeas}

\subsection{Uniform hyperbolicity -- stable and unstable foliations}

Consider now a diffeomorphism $f\colon M\to M$, where $M$ is a Riemannian manifold.  For general background on the theory of smooth dynamical systems, we refer to~\cite{KH95} and~\cite{BP07}; here we will assume that at least the basic definitions are known.

If $\Lambda \subset M$ is a hyperbolic set for $f$, then we are guaranteed the existence of local and global stable and unstable manifolds at each point $x\in\Lambda$.  The local manifolds are characterised as containing all points whose orbit converges to that of $x$ under forward or backward iteration, without ever being too far away:
\[
W_{x,\epsilon}^s = \left\{\, y\in M \Bigm| \lim_{n\to+\infty} d(f^n y,
f^n x) = 0 \text{ and } d(f^n y, f^n x) < \epsilon\ \forall n\geq 0
\,\right\}
\]
and similarly for $W_{x,\epsilon}^u$, with $n\to-\infty$ and $n\leq 0$.  For example, the set $X$ depicted in Figure~\ref{fig:cantor-shift} can be realised as a hyperbolic set for a diffeomorphism; in this case $W_{x,\epsilon}^s$ is contained in the vertical line through $x$, while $W_{x,\epsilon}^u$ is contained in the horizontal line through $x$.

The global manifolds are characterised similarly, without the requirement that the orbits always be close:
\[
W_x^s = \left\{\, y\in M \Bigm| \lim_{n\to+\infty}d(f^n y, f^n x)
= 0 \,\right\}
\]
Again, for $W_x^u$, the limit is taken as $n\to-\infty$.

The local manifolds are embedded images of Euclidean space; the global manifolds, however, are usually only immersed, and have a somewhat strange global topology.\footnote{Hence the terminology ``strange attractor'' which we see in conjuction with various dissipative systems such as the H\'enon map.} For example, they are dense in $\TT^2$ for the Anosov diffeomorphism given by the action of $\left( \begin{smallmatrix} 2 & 1 \\ 1 & 1 \end{smallmatrix} \right)$, and hence cannot be embedded images.

The connection with the previous two sections comes when we observe that given two points $x,y\in M$, either $W_x^s \cap W_y^s = \emptyset$ or $W_x^s = W_y^s$, and similarly for the unstable manifolds.  It follows that the global stable manifolds form a partition of some invariant set $X^- \supset \Lambda$; we denote this partition into global stable manifolds by $\Pi^-$, and its counterpart, the partition (of some set $X^+$) into global unstable manifolds, by $\Pi^+$.

For the linear toral automorphism mentioned above, these partitions are exactly the same as the partition into orbits of the irrational linear flow in Example~\ref{eg:nonmeas}, and we saw there that such partitions are non-measurable.  In fact, such behaviour is quite common.

\begin{theorem}\label{thm:zeromeas}
Given a $C^2$ diffeomorphism $f\colon M\to M$ and a hyperbolic set $\Lambda \subset M$, the following are equivalent:
\begin{enumerate}
\item $h_\mu(f)=0$; 
\item $\Pi^-$ is measurable; 
\item $\Pi^+$ is measurable.
\end{enumerate}
\end{theorem}

Before outlining the proof of Theorem~\ref{thm:zeromeas}, we briefly describe how one can produce many examples where the equivalent conditions all hold.  If $f$ is a hyperbolic automorphism of the two-dimensional torus and $(X,T,\nu)$ is any zero-entropy ergodic measure-preserving transformation, then it was shown in~\cite{LT77} that there exists an $f$-invariant measure $\mu$ such that the support of $\mu$ is the entire torus and $(M,f,\mu)$ is isomorphic to $(X,T,\nu)$.  For such a measure, all three conditions above hold.

\begin{proof}[Sketch of proof of Theorem~\ref{thm:zeromeas}]
We outline a proof which is due to Sinai in the case of absolutely continuous $\mu$, and in the general case  can be found in~\cite{LY85b}.\footnote{In fact, the argument has been known as ``folklore'' since the 1960's, but probably had not appeared in print before the Ledrappier--Young paper.}

Without loss of generality, assume $\mu$ is ergodic; we will sketch the construction of a \emph{leaf-subordinated partition.}
\begin{definition}[\cite{BP07}, Theorem 9.4.1]
A \emph{leaf-subordinated partition} associated with the global stable manifolds is a measurable partition $\xi$ such that
\begin{enumerate}
\item For $\mu$-a.e.\ $x$, the element of $\xi$ containing $x$ is an
open subset of $W^{s,\epsilon}(x)$ for some $\epsilon>0$ (hence in particular, $\xi\geq \Pi^-$);
\item $f\xi \geq \xi$ ($\xi$ is increasing);
\item $\xi_f = \eps$;
\item $\Pi(\xi) = \HHH(\Pi^-)$.
\end{enumerate}
\end{definition}

Conditions (3) and (4) guarantee that the increasing sequence of partitions $\xi_N^-$ has $\eps$ as one limit and $\HHH(\Pi^-)$ as the other.

Once such a partition is obtained, one proves the following lemma:

\begin{lemma}\label{lem:hisH}
For any leaf-subordinated partition $\xi$ associated with the
global stable manifolds, we have
\[
h_\mu(f) = H_\mu(f\xi | \xi).
\]
\end{lemma}
\begin{proof}
Corollary 5.3 in~\cite{LY85a}.
\end{proof}
Finally, one must show that $H_\mu(f\xi | \xi) = 0$ if and only if $\Pi^-$ is measurable; the result for $\Pi^+$ follows upon considering $f^{-1}$.

\emph{Step 1}.  To fill in some of the details of this outline, we turn first to the question of existence of leaf-subordinated partitions; this is Lemma 3.1.1 in~\cite{LY85a}, and Theorem 9.4.1 in~\cite{BP07} (although the latter deals only with the case where $\mu$ is absolutely continuous).

To construct $\xi$, divide the manifold into \emph{rectangles} -- that is, domains which exhibit the local
product structure of the manifold.\footnote{Such rectangles are of critical importance in the construction of \emph{Markov partitions}, a key tool in relating smooth dynamics to symbolic dynamics.}  More precisely, a rectangle is a domain $X\subset M$ which admits a diffeomorphism $\phi\colon X \to [0,1]^N$ such that the connected component of $\phi(W^{s,\epsilon}(x)\cap
X)$ containing $x$ is given by the set of points in $[0,1]^N$ whose first $N-k$ coordinates match those of $\phi(x)$, and similarly for $W^{u,\epsilon}(x)$, with the last $k$ coordinates matching; here $N$ is the dimension of $M$ and $k$ the dimension of the stable manifolds.

Such a partition into rectangles may be constructed in a variety of ways -- for example, by using a triangulation of the manifold $M$. Further, a standard argument along the lines of Exercises~\ref{ex:uncountablesum}--\ref{ex:basis} allows us to assume that the boundary of each rectangle has measure zero.

Now consider the partition $\xi_0$ whose elements are connected components of the stable manifolds $W^s$ intersected with a rectangle. This guarantees part of the first property, that our partition is a refinement of $\Pi^-$; to obtain an expanding partition, pass to the further refinement
\[
\xi \coloneqq \bigvee_{n=0}^\infty f^{-n} \xi_0
\]
which may be denoted $\xi=(\xi_0)^-$ using our earlier notation.  Thus $\xi$ satisfies property (2).

To see that almost every element of $\xi$ contains a ball in $W^s$, we must be slightly more careful in our construction of the rectangles, choosing them so that the measure of a $\delta$-neighbourhood of the boundary decreases exponentially with $\delta$.  Using this fact, and the fact that $\xi_0$ refines $\Pi^-$ so that the size of elements in $f^{-n}\xi_0$ grows exponentially, it is possible to show that typical elements of $\xi_0$ are only cut finitely many times during the refinement into $\xi$, which establishes property (1).

Because $\xi$ is a refinement of the partition into stable manifolds, we may bound the diameter of elements of $f^n\xi$ from above, and the bound is exponentially decreasing in $n$.  Thus $\xi_T = \bigvee_{n=0}^\infty f^n\xi = \eps$, the partition into points, so (3) holds, and we obtain (4) similarly, using the fact that $f^{-1}$ expands elements of $\xi$ exponentially along the leaves $W^s$, and so $\Pi(\xi) \coloneqq \bigwedge_{n=0}^\infty f^{-n} \xi = \HHH(\Pi^-)$.  Thus $\xi$ is the leaf-subordinated partition we were after.

\emph{Step 2}.  Now we want to describe the entropy of $f$ in terms of the entropy of $\xi$; this is accomplished by Lemma~\ref{lem:hisH}.

Regarding the proof of this lemma, recall from basic entropy theory that if $\eta$ is a finite or countable partition with $\eta_f=\eps$, then we say that $\eta$ is a \emph{generating partition}, and we have
\[
h_\mu(f) = h_\mu(f,\eta) = H_\mu(f\eta|\eta^-) = H_\mu(f\eta^-|\eta^-);
\]
thus the result would follow if $\xi_0$ was finite or countable, since $\xi=(\xi_0)^-$.  However, $\xi_0$ is continuous, so its elements have zero measure, and we cannot use this argument directly.  In the uniformly hyperbolic case, we can simply use the finite partition $\eta$ into rectangles, which refines to $\xi_0$
under iterations of $f^{-1}$.  In the general setting (for in fact versions of this theorem are true beyond the uniformly hyperbolic case), one needs a more subtle argument, as given in~\cite{LY85a}.

For a finite generating partition $\eta$, a basic result from entropy theory says that
\begin{equation}\label{eqn:generatingPinsker}
\Pi(\eta) = \pi_\mu(f),
\end{equation}
the Pinsker partition, and so if $\eta^- = \xi$, property (4) of a leaf-subordinated partition guarantees that
\[
\HHH(\Pi^-) = \pi(f),
\]
that is, that the Pinsker partition is the measurable hull of the partition into global unstable manifolds.  The fact that~\eqref{eqn:generatingPinsker} holds in general is~\cite[Theorem B]{LY85a} (stated there in terms of the associated $\sigma$-algebras), and so $\Pi^-$ is measurable if and only if it is equivalent mod zero to the Pinsker partition.

\emph{Step 3.}  With Lemma~\ref{lem:hisH} in hand, note that $h_\mu(f) = H_\mu(f\xi|\xi)=0$ if and only if $H_\mu(f^n\xi|\xi)=0$ for any (all) $n\geq 0$, and recall that if any element of $\xi$ is split into two elements of positive conditional measure in $f^n\xi$, then information is gained and the conditional entropy is positive.  Since we have an exponentially decreasing upper bound on the size of elements in $f^n\xi$, we see that if $\mu$ is not atomic, then there exists $n$ such that the refinement $f^n\xi\vee\xi$ splits some partition element of positive measure into two (or more) elements of positive measure, which guarantees $H_\mu(f^n\xi|\xi)>0$, and hence $h_\mu(f)>0$.

Thus if $h_\mu(f)=0$, then $\mu$ is atomic, with at most one atom in each element of $\xi$.  In this case, the set of atoms on each leaf $W^s$ is discrete, and since a discrete set gets denser (rarer) under forward (backward) iteration, and the measure $\mu$ is invariant, one can see that each leaf has at most one atom, and
so the conditional measures are in fact $\delta$-measures.\footnote{One must work slightly harder to show
that the conditional measure cannot be atomic with dense support -- in this case the idea is to focus on the \emph{big} atoms.} In particular, taking the union of the supports of these $\delta$-measures, we have a set of full measure which intersects each leaf exactly once (the so-called \emph{Fubini's nightmare}), and hence $\Pi^-$ is equivalent mod zero to the point partition $\eps$, which in the zero entropy case is also the Pinsker
partition $\pi(f)$.  Hence $h_\mu(f)=0$ implies that $\Pi^-$ is measurable.


It remains to prove the implication in the other direction, that measurability of $\Pi^-$ implies zero entropy.
Suppose $\Pi^-$ is measurable; then we have a system of conditional measures on global stable leaves, and each measure is finite.  The main idea is to argue that if $h_\mu(f)>0$, we may obtain arbitrarily small bounds on the conditional measure of any element of $\xi$, which will then show that all such elements have conditional measure zero.  This is a contradiction since countably many of them cover each global stable leaf, which has positive measure.

Let us make this more explicit: for a given $x$, let $C_n(x) \in f^{-n}\xi$ denote the element of $f^{-n}\xi$ containing $x$, and define conditional information functions $I_n$ by
\[
I_n(x) = -\log \mu_{C_{n+1}(x)} (C_n(x)),
\]
as in~\eqref{eqn:condinfo}, for which
\[
H_\mu(f^{-(n+1)}\xi | f^{-n}\xi ) = \int_X I_n(x) \,d\mu(x)
\]
For any $n$, the left hand side is equal to $h_\mu(f,\xi)=h_\mu(f)$, and so we see that $h_\mu(f)=\int_X
I_n(x) \,d\mu(x)$.  Further, it is apparent that
\[
\sum_{k=0}^{n-1} I_k(x) = -\log \mu_{C_n(x)} (C_0(x))
\]
and that $\mu_{C_n(x)}$ converges weakly to $\mu_{W^s(x)}$, where the latter comes from the system of conditional probability measures on global stable leaves, which exists by the assumption that the partition into
global stable leaves is measurable.  So to obtain our contradiction, we need only show that $\sum_{k=0}^\infty I_k(x)$ diverges unless $I_k$ vanishes almost everywhere.

How are the $I_k$ related to each other?  Note that given a system $\{\mu_C\}_C$ of conditional measures on elements of $f^{-n}\xi$, the pullback $\{f^*\mu_C\}_C$ is a system of conditional measures on elements of $f^{-(n+1)}\xi$, with respect to which the conditional information functions coincide.  However, since conditional measures are unique up to a constant, we do in fact have $I_{k+1}=I_k$, and the result follows.
\end{proof}

As an aside, note that the key property of $W^s$ that was used in the proof of Lemma~\ref{lem:hisH} was uniform contraction along its leaves. In general, we could take $W$ to be \emph{any} uniformly contracting partition, and we would have a version of the lemma with equality replaced by the inequality
\[
H_\mu(f\xi|\xi) \leq h_\mu(f).
\]
Taking $W$ to be the foliation in a single stable direction (say a subspace corresponding to a negative Lyapunov exponent), this allows us to speak of the contribution made by certain directions (or equivalently, certain Lyapunov exponents) to the entropy.

\subsection{Conditional measures on global leaves}

Conditional measures are a very useful tool, and we would like to use them on the partitions $\Pi^{\pm}$.  However, the theorem on existence of conditional measures only applies to \emph{measurable} partitions, and as we have seen, the partitions into global stable or unstable manifolds are only measurable in the zero entropy case.  Thus for systems with positive entropy, we cannot apply the theorem directly; however, by restricting our attention to a small section of the manifold, a rectangle, we may consider conditional measures on $W^s$ and $W^u$ within that domain.

Of course, we could choose another rectangle, which may overlap the first, and obtain conditional measures there as well; how will these two sets of conditional measures relate on the intersection?

The answer is as simple as we could hope for, and is best visualised by considering two subsets $A,B \subset X$ of positive measure with nontrivial intersection. Conditional measures $\mu_A$ and $\mu_B$ are defined in the obvious way, as the normalised restriction of $\mu$ to the appropriate domain, and it is easy to see that given $E\subset A\cap B$, we have
\[
\mu_A(E) = \frac{\mu(E)}{\mu(A)} = \frac{\mu(B)}{\mu(A)} \frac{\mu(E)}{\mu(B)} 
= \frac{\mu(B)}{\mu(A)} \mu_B(E).
\]
That is, $\mu_A$ and $\mu_B$ are proportional to each other; a similar result holds for conditional measures on stable and unstable manifolds.  If $\mu_{W^s(x)}$ and $\tilde{\mu}_{W_s(x)}$ are two families of conditional measures on stable manifolds coming from different rectangles, then they are proportional on the intersection of the two rectangles.  However, because the conditional measure on each leaf is normalised, the constant of proportionality may vary from leaf to leaf.

In this way we may define a $\sigma$-finite measure on each leaf, by gluing together conditional measures on rectangles, a procedure that is important for certain constructions in rigidity theory.

\begin{figure}[tbp]
	\includegraphics{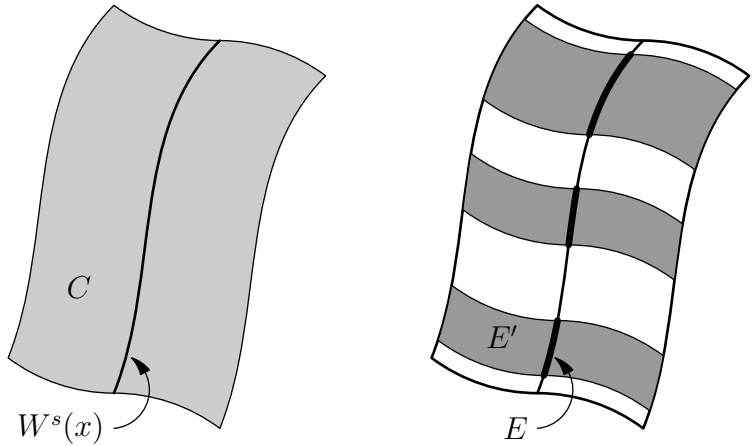}
	\caption{A geometrically intuitive interpretation of conditional measures.}
	\label{fig:cond-meas}
\end{figure}

We may think of the conditional measure on a leaf $W^s(x)$ as being the result of a limiting process.  Having fixed a rectangle, we have a product structure, and may consider small cylinders $C$ around the leaf, whose cross-sections are transversal to the leaf, as shown in Figure~\ref{fig:cond-meas}.  Given a set $E\subset W^s(x)$, then, we may consider its product $E'$ with this transversal cross-section, and approximate $\mu_{W^s(x)}(E)$ by $\mu(E')/\mu(C)$.  We would like to say that in the limit as the size of the cross-section goes to zero, this quantity converges to the conditional measure.

The one caveat regarding this interpretation is that just as $\mu_C$ is only defined for almost every leaf, so also the limit is only guaranteed to exist on almost every leaf,\footnote{Compare this with the statement of the Lebesgue density theorem, that almost every point is a density point for a given measure.} and so it may fail for the particular leaf we are interested in at a given time.  This is a manifestation of the fact that even if $\mu$ itself is rather ``nice'', the conditional measures may have very irregular dependence on the transversal direction.

\subsection{Non-uniform hyperbolicity, the Pesin Entropy Formula, and the Ledrappier--Young Theorem}

In the non-uniformly hyperbolic setting (that is, when the system has only non-zero Lyapunov exponents at almost every point), all of the above results go through more or less unchanged, with the caveat that now the structure of the foliations is intimately dependent on the measure.  (A complete description is given in~\cite{BP07}.)  We are only guaranteed existence of $W^s$ and $W^u$ at $\mu$-a.e.\ point, and there are some extra technical difficulties in the construction of $\xi^-$, which we shall not get into here.

The word ``foliation'' must be used guardedly in this setting; here it refers to a family of immersed manifolds which vary continuously in the tranvserse direction when we restrict to particular compact subsets (the Pesin sets).  On these sets, all estimates are uniform, but the Pesin sets themselves are not invariant.

Given a diffeomorphism $f\colon M\to M$ and an $f$-invariant measure $\mu$, the Multiplicative Ergodic Theorem of Oseledets guarantees the existence of the Lyapunov exponents $\chi_1 (x) < \cdots < \chi_k (x)$ at almost every point, along with the corresponding subspaces $E_1(x)\subset\cdots\subset E_k(x)=T_x M$.  These geometric quantities give infinitesimal rates of expansion and contraction, and are independent of the measure; however, if $\mu$ is ergodic then they are constant a.e., and so we may speak of the Lyapunov exponents of an ergodic measure without fear of ambiguity.

The following fundamental inequality, which states that the entropy is bounded above by the sum of the positive Lyapunov exponents, is due  to Margulis (for absolutely continuous measures) and Ruelle (in the general case):

\begin{theorem}
If $f\colon M\to M$ is a $C^1$ diffeomorphism of a smooth compact Riemannian manifold preserving a Borel probability measure $\mu$, and $d_i(x) = \dim(E_i(x))$ is the multiplicity of the $i^\text{th}$ Lyapunov exponent at $x$, then
\begin{equation}\label{eqn:Ruelles}
h_\mu(f) \leq \int \sum_{\chi_i(x) > 0} d_i(x) \chi_i(x)\,d\mu(x).
\end{equation}
\end{theorem}
\begin{proof}
Theorem 10.2.1 in~\cite{BP07}.
\end{proof}

Pesin gave conditions under which equality holds.

\begin{theorem}[Pesin Entropy Formula]
If in addition to the above hypotheses we have that $f$ is $C^{1+\alpha}$ and $\mu$ is absolutely continuous, then
\begin{equation}\label{eqn:Pesin}
h_\mu(f) = \int \sum_{\chi_i(x) > 0} d_i(x) \chi_i(x) \,d\mu(x).
\end{equation}
\end{theorem}
\begin{proof}
Theorem 10.4.1 in~\cite{BP07}.
\end{proof}

Note that the integral in~\eqref{eqn:Ruelles} and~\eqref{eqn:Pesin} is the exponential rate of volume expansion
in the unstable direction, and may also be written as
\[
\int_X \log |J_x^u f | \,d\mu(x),
\]
where $J_x^u \coloneqq J_x|_{W_x^u}$ is the Jacobian on the unstable manifold.

The proof of Pesin's entropy formula relies on the construction of leaf-subordinated partitions outlined in the proof of Theorem~\ref{thm:zeromeas}. The key step is to show that the conditional measures on $W^u$ are absolutely continuous, which allows one to establish bounds on the rate at which the volume of elements in the refined partitions decreases.

In fact, it turns out that no particular regularity of $\mu$ in the stable direction is required for Pesin's entropy formula to hold, which led Ledrappier and Young to prove the following:

\begin{theorem}[Ledrappier-Young]\label{thm:LY}
Let $f\colon M\to M$ be a $C^2$ diffeomorphism of a compact Riemannian manifold $M$ preserving a Borel probability measure $\mu$.  Then $\mu$ has absolutely continuous conditional measures on unstable manifolds if and only if~\eqref{eqn:Pesin} holds.
\end{theorem}
\begin{proof}
Theorem A in~\cite{LY85a}.
\end{proof}

A measure $\mu$ satisfying the conditions of Theorem~\ref{thm:LY} is called an \emph{SRB measure}, after Sinai, Ruelle, and Bowen.  Despite having absolutely continuous conditional measures on unstable manifolds, such measures are generally singular on $M$.

SRB measures may or may not exist for a particular system; however, in the Anosov case, they always exist, and in fact one obtains two SRB measures, one corresponding to forward iterations (which is a.c.\ in the unstable direction), and one corresponding to backward iterations (which is a.c.\ in the stable direction).  The two coincide if and only if they are absolutely continuous on $M$.

In fact, Ledrappier and Young proved a more general theorem than Theorem~\ref{thm:LY}; in~\cite{LY85b}, they show that~\eqref{eqn:Pesin} holds for arbitrary measures $\mu$, when the multiplicities $d_i(x)$ are replaced with coefficients $\delta_i^\mu$, which depend on the geometry of $\mu$ along the various foliations corresponding to different Lyapunov exponents, but which have no explicit dependence on the dynamics, despite the fact that $h_\mu(f)$ is a dynamical quantity.

For the largest Lyapunov exponent, the coefficient $\delta_n^\mu$ represents the Hausdorff dimension of the conditional measures on the corresponding foliation.  However, this does not extend to intermediate exponents, as shown by a counterexample due to Ruelle and Wilkinson, for which~\eqref{eqn:Pesin} holds, but the conditional measure in the \emph{slow} unstable direction is atomic, and so the foliation is singular~\cite{RW01}.  A formula for these coefficients may be found in Theorem 14.1.18 of~\cite{BP07}.



\end{document}